\documentclass[11pt]{amsart}

\usepackage[utf8]{inputenc}
\usepackage[a4paper,portrait,top=2cm, bottom=2cm, left=2cm, right=2cm]{geometry}
\usepackage{amsmath,amsfonts,amssymb}
\usepackage[foot]{amsaddr} 
\usepackage{enumitem}
\usepackage{cite}
\usepackage{multibib}

\usepackage
{hyperref}
\hypersetup{colorlinks=true,citecolor=blue,linkcolor=blue,urlcolor=blue, breaklinks=true}

\makeatletter 
\g@addto@macro\bfseries{\boldmath}
\makeatother

\theoremstyle{plain}
\newtheorem{theorem}{Theorem}[section]
\newtheorem{proposition}[theorem]{Proposition}

\newtheorem{lemma}[theorem]{Lemma}

\newtheorem{conjecture}[theorem]{Conjecture}

\theoremstyle{remark}

\theoremstyle{definition}
\newtheorem{remark}[theorem]{Remark}
\newtheorem{convention}[theorem]{Convention}

\def\Z{\mathbb{Z}}
\def\Q{\mathbb{Q}}

\def\O{\mathcal{O}} 
\def\P{\mathcal{P}} 

\newcommand{\BQ}[2]{\Q\big(\!\sqrt{#1}, \sqrt{#2}\big)} 
\newcommand{\K}[2]{K_{{#1},{#2}}} 

\newcommand{\zeros}{0^{\infty}}
\newcommand{\evens}{\text{even}^{\infty}}

\newcommand{\flfl}[1]{\lfloor\!\lfloor{#1}\rfloor\!\rfloor} 
\newcommand{\cece}[1]{\lceil\!\lceil{#1}\rceil\!\rceil} 

\newcommand{\flflsq}[1]{\flfl{\sqrt{#1}}} 
\newcommand{\cecesq}[1]{\cece{\sqrt{#1}}} 

\DeclareMathOperator{\Tr}{Tr}

\begin{document}

\title{On biquadratic fields: when 5 squares are not enough}

\author[D. Dombek]{Daniel Dombek}



\address{D. Dombek\newline
\indent Department of Applied Mathematics\newline
\indent Faculty of Information Technology, CTU in Prague\newline
\indent Thákurova 9\newline
\indent 160 00 Prague 6, Czech Republic}

\email{daniel.dombek@fit.cvut.cz}


\begin{abstract} 
  In this paper we study the Pythagoras number $\P(\O_K)$ for the rings of integers in totally real biquadratic fields $K$. We continue the work of Tinková towards proving the conjecture by Krásenský, Raška and Sgallová that a biquadratic $K$ satisfies $\P(\O_K)\geq 6$ if and only if it contains neither $\sqrt{2}$ nor $\sqrt{5}$, with only finitely many exceptions. We fully solve two out of three remaining classes of fields by proving that all but finitely many $K$ containing $\sqrt{6}$ or $\sqrt{7}$ satisfy $\P(\O_K)\geq 6$. Furthermore, we present ideas and computations which further support the conjecture also for $K$ containing $\sqrt{3}$. This enables us to refine the conjecture by explicitly listing the exceptional fields.
\end{abstract}

\maketitle


\section{Introduction}

  In this paper, we study the so-called \emph{Pythagoras number} of rings of integers in totally real biquadratic fields, denoted by $\P(\O_K)$. We follow the works of~\cite{KraRasSga22,Tin25,HeHu22} in an effort to classify all biquadratic fields according to their $\P(\O_K)$ and we come one step closer to proving the final conjecture from~\cite{KraRasSga22}. The structure of this paper is as follows: The rest of the Introduction deals (very briefly) with prerequisites and gives a summary of the state-of-the-art. Section~\ref{sec:main} presents the main theorems and conjecture. The Theorems~\ref{thm:main1} and \ref{thm:main2} are then proven in Section~\ref{sec:proofs}. In Section~\ref{sec:open}, we discuss the open problems regarding $\P(\O_K)$ for biquadratic fields.

  Given any commutative ring $R$, denote by $\sum R^2$ its subset containing all (finite) sums of squares. For each $\alpha \in \sum R^2$, the minimal number of squares whose sum is equal to $\alpha$ is called its \emph{length} and denoted $\ell(\alpha)$. The \emph{Pythagoras number} is then defined as 
    \[
    \P(R)=\sup\{\ell(\alpha)\,:\,\alpha \in \textstyle\sum R^2\}\,,
    \]
  i.e.~as the smallest number $n$ with the property that any $\alpha\in\sum R^2$ can be written as a sum of at most $n$ squares. In recent years, there has been a development in the study of the Pythagoras numbers of orders in algebraic number fields~\cite{DaaGajSiuYat25,GilTin25,HeHu22,KalSgaTin25,KalYat21,Kra22,KraRasSga22,Tin23,Tin25}. For an overview of the subject, the reader may refer e.g.~to the Introduction in~\cite{KraRasSga22}.

  We study the rings of integers $\O_K$ in \emph{totally real biquadratic fields} $K$. These are the fields $K=\BQ{p}{q}$, where $p,q>1$ are distinct square-free rational integers (note that we often write just ``biquadratic field''). 
    
  In any biquadratic $K$, there exist exactly three distinct square-free rational integers $p,q,r>1$ which are squares in $K$. They satisfy $K=\BQ{p}{q}=\BQ{q}{r}=\BQ{p}{r}$. Moreover, one can be obtained from the other two as their product divided by the square of their $\gcd$, e.g.~$r=\frac{pq}{\gcd(p,q)^2}$. For our purposes, it is crucial to know the integral basis of $\O_K$. It is well known (due to~\cite{Wil70}), that it is of one of the following types (B1--B4), depending on $p,q,r$:

  \begin{tabular}{lll}
    (B1) & $\O_K=\mathrm{Span}_{\Z}\Bigl\{1, \sqrt{p}, \sqrt{q}, \frac{\sqrt{p}+\sqrt{r}}2\Bigr\}$ & if $(p,q) \equiv (2,3) \pmod{4}$,\\
    (B2) & $\O_K=\mathrm{Span}_{\Z}\Bigl\{1, \sqrt{p}, \frac{1+\sqrt{q}}2, \frac{\sqrt{p}+\sqrt{r}}{2} \Bigr\}$ & if $(p,q) \equiv (2,1) \pmod{4}$,\\
    (B3) & $\O_K=\mathrm{Span}_{\Z}\Bigl\{1, \sqrt{p}, \frac{1+\sqrt{q}}2, \frac{\sqrt{p}+\sqrt{r}}{2}\Bigr\}$ & if $(p,q) \equiv (3,1) \pmod{4}$,\\
    (B4a) & $\O_K=\mathrm{Span}_{\Z}\Bigl\{1, \frac{1+\sqrt{p}}2,\frac{1+\sqrt{q}}2, \frac{1+\sqrt{p}+\sqrt{q}+\sqrt{r}}{4}\Bigr\}$ & if $(p,q,\gcd(p,q)) \equiv (1,1,1) \pmod{4}$,\\
    (B4b) & $\O_K=\mathrm{Span}_{\Z}\Bigl\{1, \frac{1+\sqrt{p}}2,\frac{1+\sqrt{q}}2, \frac{1-\sqrt{p}+\sqrt{q}+\sqrt{r}}{4}\Bigr\}$ & if $(p,q,\gcd(p,q)) \equiv (1,1,3) \pmod{4}$.
  \end{tabular}

  It is not hard to see that every totally real biquadratic field belongs to exactly one of these types (after possibly interchanging $p,q,r$).

  Throughout this paper, we will however adhere to a different convention for denoting biquadratic fields. It was already used in~\cite{KraTinZem20,KraRasSga22} and it uses a triple of non-interchangeable letters $m,s,t$, instead of $p,q,r$ used before. We also adopt a more compact notation for $K=\BQ{p}{q}$.

  \begin{convention}\label{con:mst}
    Let $K$ be a biquadratic field. We denote by $m,s,t$ the only three square-free rational integers greater than $1$ which are squares in $K$, which moreover satisfy
      \[
      1 < m < s < t.
      \]
    In the following, we will always denote the field $K$ as $\K{m}{s}$ where the subscripted $m,s$ are \emph{always} the smallest two out of $m,s,t$, moreover in the increasing order.
  \end{convention}

  It will be important to keep this convention in mind throughout the paper. For instance, if we write $\K{6}{s}$, it means that $s>6$ is square-free, $\frac{6s}{\gcd(6,s)^2}>s$, and $K=\BQ{6}{s}$. Hence $K$ is neither $\BQ{6}{2}=\BQ{6}{3}$ nor $\BQ{6}{5}$.

  As a vector space over $\Q$, the field $\K{m}{s}$ has basis $(1,\sqrt{m},\sqrt{s},\sqrt{t})$. In this basis, it is particularly easy to compute the \emph{trace} $\Tr:K\to\Q$, as it satisfies $\Tr(a+b\sqrt{m}+c\sqrt{s}+d\sqrt{t})=4a$. We conclude this section by summarizing those results on the topic of $\P(\O_K)$ for $K=\K{m}{s}$ which are most relevant to us.

  \begin{theorem}[Pythagoras number of integers in biquadratic fields: state-of-the-art]\label{thm:state}
    Let $K$ be a totally real biquadratic field.
    \begin{enumerate}[label= (\alph*)]
      \itemsep 2mm
      \item~\cite{KalYat21,KraRasSga22} $\P(\O_K)\leq 7$ and there exist infinitely many $K$ with $\P(\O_K)=7$.
      \item~\cite{Tin25,KraRasSga22} If $K$ does not contain any of $\sqrt{2}, \sqrt{3}, \sqrt{5}, \sqrt{6}, \sqrt{7}$, then $\P(\O_K)\geq 6$.\footnote{This is almost exactly the main result of~\cite{Tin25}, which has one additional assumption that $K$ does not contain $\sqrt{13}$ either. However, fields $\K{13}{s}$ have already been dealt with in~\cite[Section 8]{KraRasSga22}.}
      \item~\cite{KraRasSga22,HeHu22} If $K$ contains $\sqrt{2}$ or $\sqrt{5}$, then $\P(\O_K) = 5$ with at most six exceptions (for which $\P(\O_K) \leq 5$).
      \item~\cite{KraRasSga22} If $K \neq \K{3}{7}$ contains $\sqrt{3}$, $\sqrt{6}$ or $\sqrt{7}$ and does not belong to the exceptions in (c), then $\P(\O_K) \geq 5$.  
    \end{enumerate}
  \end{theorem}

  As can be seen from Theorem~\ref{thm:state}, the exact division line between biquadratic fields with $\P(\O_K)\leq 5$ and $\P(\O_K)\geq 6$ is not known yet. The problematic class of fields are those containing $\sqrt{3}$, $\sqrt{6}$ or $\sqrt{7}$. The authors of~\cite{KraRasSga22} proposed a conjecture that they in fact belong to the latter group, which was supported by several partial results and computer experiments. They conjectured that:

  \begin{itemize}
    \itemsep 1mm
    \item $\P(\O_K)=3$ for $\K{2}{3}$, $\K{2}{5}$,\footnote{The field $\K{2}{5}$ is indeed already known to satisfy $\P(\O_K)=3$, thanks to the upcoming paper by J.~Krásenský and R.~Scharlau. For a sketch of the proof, see the Introduction in~\cite{Kra23}.} $\K{3}{5}$,
    \item $\P(\O_K)=4$ for $\K{2}{7}$, $\K{3}{7}$, $\K{5}{6}$ and $\K{5}{7}$,
    \item $\P(\O_K)=5$ all fields containing $\sqrt{2}$ or $\sqrt{5}$ except those listed above,
    \item $\P(\O_K)\geq 6$ for the remaining fields, with only finitely many exceptions (for which $\P(\O_K)=5$).
  \end{itemize}

  We focus on the last point. In the light of Theorem~\ref{thm:state}(b), the only part left to prove concerns the fields $\K{3}{s}$, $\K{6}{s}$ and $\K{7}{s}$. It is known for sure that they satisfy $\P(\O_K)\geq 5$ up to finitely many exceptions, see Theorem~\ref{thm:state}(d). 
  
  In the following, we will improve this to $\P(\O_K)\geq 6$ for all but finitely many $\K{6}{s}$ and $\K{7}{s}$ and provide more data and ideas for the conjecture that $\P(\O_K)\geq 6$ holds also for all but finitely many $\K{3}{s}$. However, this family of fields is quite hard to handle -- it seems challenging to prove even the weaker statement that infinitely many fields $\K{3}{s}$ satisfy $\P(\O_K)\geq 6$.

  \begin{remark}
    Remark~\ref{rem:support}, Lemma~\ref{lem:m67pc} and Propositions~\ref{pr:m3s1}--\ref{pr:m3s3} are based on computer calculations. In particular, we used:
    \begin{enumerate}[label= (\roman*)]
      \item The publicly available program written by M.~Raška in Python at \url{https://github.com/raskama/number-theory/tree/main/biquadratic} (see \cite[Section 4.2]{KraRasSga22} for more details), which systematically computes lower bounds for $\P(\O_K)$ for fields with small $m,s$. Namely, it computes the lengths of all elements with trace bounded by a given constant.
      \item Several scripts of our own in Wolfram Mathematica for various calculations. All of them were connected to verifying that a given element has length $6$. This is possible, since for any fixed $\alpha\in\O_K$, there is only finitely many squares to be considered as a possible summand. These calculations were independently verified by a script in Magma Computational Algebra System \cite{BosCanPla97}.
    \end{enumerate}    
  \end{remark}


\section{Main results}\label{sec:main}

  \begin{theorem}\label{thm:main1}
    Let $K$ be a totally real biquadratic field containing none of $\sqrt{2}$, $\sqrt{3}$, $\sqrt{5}$. Then $\P(\O_K) \geq 6$ except possibly for $K=\K{6}{14}$.
  \end{theorem}

  \begin{proof}
    As follows from~\cite{Tin25,KraRasSga22}, which we summarize in Theorem~\ref{thm:state}(b), the statement holds for all $K$ not containing any of $\sqrt{2}$, $\sqrt{3}$, $\sqrt{5}$, $\sqrt{6}$ and $\sqrt{7}$. The rest of the proof is contained in Propositions~\ref{pr:m7s1}--\ref{pr:m7s3} for $K=\K{7}{s}$ and in Propositions~\ref{pr:m6s1}--\ref{pr:m6s3} for $K=\K{6}{s}$.
  \end{proof}

  On the other hand, we know that the fields containing $\sqrt{2}$ or $\sqrt{5}$ have $\P(\O_K)\leq 5$, see Theorem~\ref{thm:state}(c). We now proceed to the remaining class of fields, i.e.~to those containing $\sqrt{3}$, for which we provide partial results.

  \begin{theorem}\label{thm:main2}
    If $K=\K{3}{s}$ with $17 \leq s \leq 511$, then $\P(\O_K) \geq 6$, with at most eight exceptions. \footnote{The possible exceptions are $\K{3}{s}$ with $s\in\{170,178,230,238,362,442,446,454\}$. We believe that $\P(\O_K)\geq 6$ will still be true here as the exceptions are probably only due to our lack of computational power.}  
    
    Moreover, if $s \equiv 1 \pmod{4}$, then $\P(\O_K)\geq 6$ holds for all $17 \leq s < 15000$, $315^2 < s < 317^2$, $999^2 < s < 1001^2$ and $3161^2 < s < 3165^2$.
  \end{theorem}

  \begin{proof}
    The proof is contained in Propositions~\ref{pr:m3s1}--\ref{pr:m3s3}.
  \end{proof}

  This gives us quite a lot of confidence to conjecture that there indeed is a clear division line between biquadratic fields with $\P(\O_K) \leq 5$ and $\P(\O_K) \geq 6$.
  \begin{conjecture}\label{conj:final}
    Let $K$ be a totally real biquadratic field. Then $\P(\O_K)\leq 5$ if and only if one the following is true:
    \begin{itemize}
      \item $K$ contains $\sqrt{2}$ or $\sqrt{5}$, 
      \item $K=\K{3}{s}$ with $s\in\{7,10,11,13,14\}$, 
      \item $K=\K{6}{14}$. 
    \end{itemize}  
  \end{conjecture}
  
  This conjecture can be refined by including the aforementioned conjecture from~\cite{KraRasSga22} that the only fields with $\P(\O_K)=3$ are $\K{2}{3}$, $\K{2}{5}$ and $\K{3}{5}$ and that the only fields with $\P(\O_K)=4$ are $\K{2}{7}$, $\K{3}{7}$, $\K{5}{6}$ and $\K{5}{7}$. 
  
  \begin{remark}\label{rem:support}
    We further support the Conjecture~\ref{conj:final} by the following computational results:
    \begin{itemize}
      \item If $K=\K{2}{3}$ or $K=\K{3}{5}$, then all $\alpha\in \sum\O_K^2$ with $\Tr(\alpha)\leq 750$ have $\ell(\alpha)\leq 5$.
      \item If $K=\K{3}{s}$ with $s\in\{7,10,11,13,14\}$, then all $\alpha\in \sum\O_K^2$ with $\Tr(\alpha)\leq 1000$ have $\ell(\alpha)\leq 5$.
      \item If $K=\K{6}{14}$, then all $\alpha\in \sum\O_K^2$ with $\Tr(\alpha)\leq 1000$ have $\ell(\alpha)\leq 5$.
    \end{itemize}
  \end{remark}


\section{Proofs}\label{sec:proofs}

  \begin{convention}\label{con:wlog}
    All the following proofs rely on similar reasoning. We start with some $\alpha\in\O_K$ which is a sum of six squares and our candidate for $\ell(\alpha)=6$. Then we study all possible representations of $\alpha$ as a sum of squares, i.e.~we assume that
      \[
      \alpha = \sum x_i^2= \sum \Bigl(\frac{a_i+b_i\sqrt{m}+c_i\sqrt{s}+d_i\sqrt{t}}{2}\Bigr)^2\,,
      \]
    with some additional parity conditions for $a_i,b_i,c_i,d_i\in\Z$ (so that each $x_i\in\O_K$). This is equivalent to a system of Diophantine equations, which we then proceed to solve.

    For greater clarity, we typically work with sequences $a=(a_i)_{i\geq 1}$ and similarly $b$, $c$ and $d$, instead of individual coefficients $a_i,b_i,c_i,d_i$. For simplicity, we denote these sequences as infinite, even though they can have only finitely many nonzero terms. The strings $\zeros$ and $\evens$ represent infinite tails of zeros and even numbers, respectively. 

    Finally, let us emphasize that the representation of a given number as a sum of squares is typically not unique. For one, both $(a_i,b_i,c_i,d_i)$ and $(-a_i,-b_i,-c_i,-d_i)$ lead to the same square $x_i^2$. For another, the squares can be reordered (which is equivalent to some simultaneous permutation of coefficients in all $a,b,c,d$). \emph{In the proofs, we will often make a choice without loss of generality which is based on these two facts (without mentioning it explicitly, for the sake of brevity).}
  \end{convention}

  We begin with the summary of all computational results necessary for the fields with $m\in\{6,7\}$ and then follow with the general proofs themselves.

  \begin{lemma}\label{lem:m67pc}
    The following elements satisfy $\ell(\alpha)=6$ in their respective fields:

    \begin{enumerate}[label= (\roman*)]
      \item $\alpha = 1^2 + 1^2 + 1^2 + (2+\sqrt{7})^2 + \Bigl(\frac{\sqrt{7}+\sqrt{11}}{2}\Bigr)^2 + \Bigl(1+\frac{\sqrt{7}+\sqrt{11}}{2}\Bigr)^2$ in $\K{7}{11}$,
      \item $\alpha = 1^2 + 1^2 + 1^2 + (2+\sqrt{6})^2 + \Bigl(\frac{\sqrt{6}+\sqrt{10}}{2}\Bigr)^2 + \Bigl(1+\frac{\sqrt{6}+\sqrt{10}}{2}\Bigr)^2$ in $\K{6}{10}$,
      \item $\alpha = 1^2 + 1^2 + 1^2 + (1+\sqrt{6})^2 + \Bigl(\frac{\sqrt{6}+\sqrt{s}}{2}\Bigr)^2 + \Bigl(1+\frac{\sqrt{6}+\sqrt{s}}{2}\Bigr)^2$ in $\K{6}{s}$ \\
      for $22 \leq s \leq 110, s \equiv 2 \pmod{4}$,
      \item $\alpha = 1^2 + 1^2 + 1^2 + (1+\sqrt{6})^2 + \Bigl(\frac{\sqrt{6}+\sqrt{6s}}{2}\Bigr)^2 + \Bigl(1+\frac{\sqrt{6}+\sqrt{6s}}{2}\Bigr)^2$ in $\K{6}{s}$ \\
      for $s\in\{7,11\}$.\footnote{Note that the case of $s=7$ is redundant (and included only for the sake of completeness) as we already know from~\cite[Lemma 5.2]{KraYat23}, that $\K{6}{7}$ has $\P(\O_K)=7$.}
    \end{enumerate}
  \end{lemma}

  \begin{proof}
    We only need to show that none of these elements can be written as a sum of five squares. In principle, this could be checked by hand. However, the most convenient way of verifying this is using the function \texttt{IsRepresented} in Magma; we also tested it independently by our simple script in Wolfram Mathematica.
  \end{proof}

  \subsection{Fields with $m=7$}\label{subsec:m7}

  With $7$ being a prime, Convention~\ref{con:mst} leads to the following simple assumptions about $s$ and $t$ for $\K{7}{s}$: $s>7$ has to be square-free and not divisible by $7$, and $t=7s$.

    \begin{proposition}\label{pr:m7s1}
      Let $K=\K{7}{s}$ for $s \equiv 1 \pmod{4}$. Then $\P(\O_K)\geq 6$.

      In particular, the following element has length $6$:
        \[
        \alpha = 1^2 + 1^2 + 1^2 + (1+\sqrt{7})^2 + \Bigl(\frac{1+\sqrt{s}}{2}\Bigr)^2 + \Bigl(1+\frac{1+\sqrt{s}}{2}\Bigr)^2.
        \]
    \end{proposition}
    
    \begin{proof}
      Suppose that $\alpha=\sum x_i^2$ for $x_i\in \O_K$. As $K$ is of type (B3) and the integral basis is $\Bigl(1,\sqrt{7}, \frac{1+\sqrt{s}}{2}, \frac{\sqrt{7}+\sqrt{7s}}{2}\Bigr)$, we may write
        \[
        \alpha = \Bigl(13 + \frac{s+1}{2}\Bigr) + 2\sqrt{7} + 2\sqrt{s} = \sum x_i^2 = \sum\Bigl(\frac{a_i + b_i\sqrt{7} + c_i\sqrt{s} + d_i\sqrt{7s}}{2}\Bigr)^2,
        \]
      where all $a_i\equiv c_i$ and $b_i\equiv d_i \pmod{2}$.

      By expanding the squares and comparing coefficients in front of $1$, $\sqrt{7}$, $\sqrt{s}$ and $\sqrt{7s}$, one gets the following conditions:
        \begin{align}
        54+2s & = \sum a_i^2 + 7\sum b_i^2 + s\sum c_i^2 + 7s\sum d_i^2,\label{eq:7s1:1}\\
        4 & = \sum a_i b_i + s\sum c_i d_i,\label{eq:7s1:2} \\
        4 & = \sum a_i c_i + 7\sum b_i d_i,\label{eq:7s1:3}\\
        0 & = \sum a_i d_i + \sum b_i c_i.\label{eq:7s1:4}
        \end{align}

      Note that, by Convention~\ref{con:mst} ($7=m<s<t$, all square-free), we actually assume $s\geq 13$. If there were a nonzero $d_i$, then $\sum d_i^2\geq 1$ and~\eqref{eq:7s1:1} would yield $54+2s \geq 7s$, which is a contradiction. So $d=(\zeros)$ and $b$ contains only even numbers.

      This simplifies the conditions above, and we get that both in~\eqref{eq:7s1:2} and~\eqref{eq:7s1:3}, the first sum has to be nonzero. Hence $\sum a_i^2, \sum c_i^2 \geq 1$ and $\sum b_i^2 \geq 4$. If $\sum c_i^2 \geq 4$, then~\eqref{eq:7s1:1} would again contradict $s \geq 13$, hence $\sum c_i^2 \in \{1,2,3\}$.

      Without loss of generality\footnote{As we warned in Convention~\ref{con:wlog}, we will in future omit the phrase ``without loss of generality'' in analogous situations.}, $c$ is then one of $(1,\zeros), (1,1,\zeros), (1,1,1,\zeros)$ and $a$ starts with one, two or three odd numbers (respectively) with all others being even. However, the first and the last case would make the right hand side of~\eqref{eq:7s1:3} an odd number, which is a contradiction. That gives us the following conditions:
        \[
        d=(\zeros), c=(1,1,\zeros), b=(\evens), a=(\text{odd},\text{odd},\evens).
        \]

      Equation~\eqref{eq:7s1:1} now reads $\sum a_i^2 + 7\sum b_i^2 = 54$ and we already know that $\sum b_i^2 \geq 4$ is divisible by $4$. Necessarily, $\sum a_i^2=26$ and $\sum b_i^2=4$, hence there exists exactly one nonzero $b_j=\pm 2$. By~\eqref{eq:7s1:4}, $b_1+b_2=0$, which leads to $j\geq 3$ and, without loss of generality, we can choose $b=(0,0,2,\zeros)$.

      It remains to determine $a=(\text{odd},\text{odd},\evens)$. It has to satisfy
        \[
        \sum a_i^2=26,\ a_3=2,\ a_1+a_2=4.
        \]

      This leads to $\{a_1,a_2\}=\{3,1\}$, $a_3=2$ and exactly three more $a_i$ equal to $2$ in modulus. Hence we have proven that if $\alpha$ is a sum of squares, then the corresponding coefficients must satisfy (without loss of generality):
        \[
        d=(\zeros), c=(1,1,\zeros), b=(0,0,2,\zeros), a=(3,1,2,2,2,2,\zeros).
        \]
      
      It follows that $\alpha$ has length $6$, as the only way to express it as a sum of squares is
        \[
        \alpha = \Bigl(\frac{3+\sqrt{s}}{2}\Bigr)^2 + \Bigl(\frac{1+\sqrt{s}}{2}\Bigr)^2 + (1+\sqrt{7})^2 + 1^2 + 1^2 + 1^2.\hfill\qedhere
        \]
    \end{proof}

    
    \begin{proposition}\label{pr:m7s2}
      Let $K=\K{7}{s}$ for $s \equiv 2 \pmod{4}$. Then $\P(\O_K)\geq 6$.

      In particular, the following element has length $6$:      
        \[
        \alpha = 1^2 + 1^2 + 1^2 + (1+\sqrt{7})^2 + \Bigl(\frac{\sqrt{s}+\sqrt{7s}}{2}\Bigr)^2 + \Bigl(1+\frac{\sqrt{s}+\sqrt{7s}}{2}\Bigr)^2.
        \]
    \end{proposition}

    \begin{proof}
      Suppose that $\alpha=\sum x_i^2$ for $x_i\in \O_K$. As $K$ is of type (B1) and the integral basis is $\Bigl(1,\sqrt{7}, \sqrt{s}, \frac{\sqrt{s}+\sqrt{7s}}{2}\Bigr)$, we may write
        \[
        \alpha = (12 + 4s) + (2+s)\sqrt{7} + \sqrt{s} + \sqrt{7s} = \sum x_i^2 = \sum\Bigl(\frac{a_i + b_i\sqrt{7} + c_i\sqrt{s} + d_i\sqrt{7s}}{2}\Bigr)^2,
        \]
      where all $a_i\equiv b_i\equiv 0$ and $c_i\equiv d_i \pmod{2}$.

      By expanding the squares and comparing coefficients, one gets the following conditions:
        \begin{align}
        48+16s & = \sum a_i^2 + 7\sum b_i^2 + s\sum c_i^2 + 7s\sum d_i^2,\label{eq:7s2:1}\\
        4+2s & = \sum a_i b_i + s\sum c_i d_i,\label{eq:7s2:2} \\
        2 & = \sum a_i c_i + 7\sum b_i d_i,\label{eq:7s2:3}\\
        2 & = \sum a_i d_i + \sum b_i c_i.\label{eq:7s2:4}
        \end{align}

      From here, we continue the proof in a similar way as for Proposition~\ref{pr:m7s1}. Note that, by Convention~\ref{con:mst}, we  assume $s\geq 10$. Together with~\eqref{eq:7s2:1}, this yields that $\sum d_i^2$ cannot be greater than $2$. The case $\sum d_i^2=0$ leads to all $a_i,b_i,c_i$ being even, which can be ruled out by considering~\eqref{eq:7s2:3} modulo $4$. If we assume $\sum d_i^2=1$, hence $d=(1,\zeros)$ and $c=(\text{odd},\evens)$, then a similar divisibility reasoning will lead to a contradiction in~\eqref{eq:7s2:2}.

      So, we get $d=(1,1,\zeros)$ and $c=(\text{odd},\text{odd},\evens)$. If we consider~\eqref{eq:7s2:2} modulo $s$, we get that $b\neq(\zeros)$, which leads to $\sum b_i^2\geq 4$.

      Thanks to $s\geq 10 $ and~\eqref{eq:7s2:1}, it follows that $\sum c_i^2$ cannot exceed $4$. So $|c_1|=|c_2|=1$ and because of the parity conditions, all remaining $c_i=0$. Moreover, one cannot have $c_1=c_2=-1$, as adding together~\eqref{eq:7s2:3}$+$\eqref{eq:7s2:4} gives $4=6(b_1+b_2)$, a contradiction. We arrive at
        \[
        d=(1,1,\zeros), c=(1,\pm1,\zeros), b=(\evens), a=(\evens).
        \]

      From~\eqref{eq:7s2:1} we get $48 = \sum a_i^2 + 7\sum b_i^2$, which yields $\sum a_i^2 = 20$ and $\sum b_i^2=4$. Hence there must be exactly one $b_j = \pm2$, while all others are zero. Consequently, one can exclude the case $c_2=-1$, it suffices to use~\eqref{eq:7s2:2}. Finally, subtracting~\eqref{eq:7s2:3}$-$\eqref{eq:7s2:4} gives us $b_1+b_2=0$, so, necessarily, those must be zero and we choose $b_3=2$.

      Updating all the values gives
        \[
        d=(1,1,\zeros), c=(1,1,\zeros), b=(0,0,2,\zeros), a=(\evens)
        \]
      and it remains to deal with $a$, satisfying $a_i\equiv 0\pmod{2}$, $\sum a_i^2=20,\ a_3=2,\ a_1+a_2=2$. This necessarily leads to $a=(2,0,2,2,2,2,\zeros)$. Thus the only way of writing $\alpha$ as a sum of squares is
        \[
        \alpha = \Bigl(1+\frac{\sqrt{s}+\sqrt{7s}}{2}\Bigr)^2 + \Bigl(\frac{\sqrt{s}+\sqrt{7s}}{2}\Bigr)^2 + (1+\sqrt{7})^2 + 1^2 + 1^2 + 1^2\,,
        \]
      in particular, $\ell(\alpha)=6$.
    \end{proof}


    \begin{proposition}\label{pr:m7s3}
      Let $K=\K{7}{s}$ for $s \equiv 3 \pmod{4}$. Then $\P(\O_K)\geq 6$.

      In particular, the following element has length $6$:
        \[
        \alpha = \begin{cases}
          1^2 + 1^2 + 1^2 + (1+\sqrt{7})^2 + \Bigl(\frac{\sqrt{7}+\sqrt{s}}{2}\Bigr)^2 + \Bigl(1+\frac{\sqrt{7}+\sqrt{s}}{2}\Bigr)^2 & \text{ if } s\geq 15,\\
          1^2 + 1^2 + 1^2 + (2+\sqrt{7})^2 + \Bigl(\frac{\sqrt{7}+\sqrt{11}}{2}\Bigr)^2 + \Bigl(1+\frac{\sqrt{7}+\sqrt{11}}{2}\Bigr)^2 & \text{ if } s=11.
        \end{cases}
        \]
    \end{proposition}

    \begin{proof}
      As the case $s=11$ is handled in Lemma~\ref{lem:m67pc}, we assume $s\geq 15$. Again, suppose that $\alpha=\sum x_i^2$ for $x_i\in \O_K$. As $K$ is of type (B3) and the integral basis is $\Bigl(1,\sqrt{7}, \frac{1+\sqrt{7s}}{2}, \frac{\sqrt{7}+\sqrt{s}}{2}\Bigr)$, we may write
        \[
        \alpha = \Bigl(15 + \frac{s+1}{2}\Bigr) + 3\sqrt{7} + \sqrt{s} + \sqrt{7s} = \sum x_i^2 = \sum\Bigl(\frac{a_i + b_i\sqrt{7} + c_i\sqrt{s} + d_i\sqrt{7s}}{2}\Bigr)^2,
        \]
      where all $a_i\equiv d_i$ and $b_i\equiv c_i \pmod{2}$.

      By expanding the squares and comparing coefficients, one gets the following conditions:
        \begin{align}
        62+2s & = \sum a_i^2 + 7\sum b_i^2 + s\sum c_i^2 + 7s\sum d_i^2,\label{eq:7s3:1}\\
        6 & = \sum a_i b_i + s\sum c_i d_i,\label{eq:7s3:2} \\
        2 & = \sum a_i c_i + 7\sum b_i d_i,\label{eq:7s3:3}\\
        2 & = \sum a_i d_i + \sum b_i c_i.\label{eq:7s3:4}
        \end{align}

      One can see from~\eqref{eq:7s3:1} and $s\geq 15$ that necessarily $\sum d_i^2=0$, so $d=(\zeros)$ and $a=(\evens)$. From~\eqref{eq:7s3:2},~\eqref{eq:7s3:3} and~\eqref{eq:7s3:4}, one immediately sees that none of $a,b,c$ can be $(\zeros)$, so $\sum a_i^2 \geq 4$ and $\sum b_i^2, \sum c_i^2 \geq 1$.

      While~\eqref{eq:7s3:1} gives us an upper bound $\sum c_i^2 \leq 5$, we can narrow it down to $\sum c_i^2=2$ in the following few steps. As all $b_i\equiv c_i\pmod{2}$, it follows from~\eqref{eq:7s3:4} that the number of odd entries in $c$ must be even. Similarly, as all $a_i$ are even, \eqref{eq:7s3:3} gives us that at least one $c_i$ must be odd. Consequently
        \[
        c\in\{(1,1,\zeros), (1,1,1,1,\zeros)\}
        \]
      and it remains to eliminate the latter. It is easy to see that $c=(1,1,1,1,\zeros)$ implies $s=15$ and $\sum a_i^2=\sum b_i^2=4$. Consequently, $b=(1,1,1,-1,\zeros)$ by~\eqref{eq:7s3:4}, and there is no solution $a$ satisfying both \eqref{eq:7s3:2} and~\eqref{eq:7s3:3} at the same time. Hence we arrive at
        \[
        d=(\zeros), c=(1,1,\zeros), b=(\text{odd}, \text{odd}, \evens), a=(\evens).
        \]

      From~\eqref{eq:7s3:1} we get $62 = \sum a_i^2 + 7\sum b_i^2$, which gives us $\sum b_i^2 \leq8$ and all $|b_i| \leq 2$. It follows from~\eqref{eq:7s3:4} that $b_1=b_2=1$.
      Moreover, $b_3\in\{0,2\}$ and all other $b_i=0$. But as $b_3=0$ leads to a contradiction between~\eqref{eq:7s3:2} and~\eqref{eq:7s3:3}, we get $b=(1,1,2,\zeros)$.

      It remains to deal with $a$, satisfying $a_i\equiv 0\pmod{2}$, $\sum a_i^2=20,\ a_1+a_2+2a_3=6,\ a_1+a_2=2$. This leads to $a=(2,0,2,2,2,2,\zeros)$. Thus the only way of writing $\alpha$ as a sum of squares is
        \[
        \alpha = \Bigl(1+\frac{\sqrt{7}+\sqrt{s}}{2}\Bigr)^2 + \Bigl(\frac{\sqrt{7}+\sqrt{s}}{2}\Bigr)^2 + (1+\sqrt{7})^2 + 1^2 + 1^2 + 1^2\,,
        \]
      in particular, $\ell(\alpha)=6$.
    \end{proof}


  \subsection{Fields with $m=6$}\label{subsec:m6}

    As $6$ is not a prime, the relations between $m,s,t$ for $\K{6}{s}$ following from Convention~\ref{con:mst} fall into two types: If $s>6$, square-free, is coprime with $6$, then $t=6s$. On the other hand, it may happen that $\gcd(6,s)>1$. Necessarily, $\gcd(6,s)=2$, $s$ is even and $t=\frac{3}{2}s$.

    \begin{proposition}\label{pr:m6s1}
      Let $K=\K{6}{s}$ for $s \equiv 1 \pmod{4}$. Then $\P(\O_K)\geq 6$.

      In particular, the following element has length $6$:
        \[
        \alpha = 1^2 + 1^2 + 1^2 + (1+\sqrt{6})^2 + \Bigl(\frac{1+\sqrt{s}}{2}\Bigr)^2 + \Bigl(1+\frac{1+\sqrt{s}}{2}\Bigr)^2.
        \]
    \end{proposition}

    \begin{proof}
      The proof is almost verbatim the same as the proof of Proposition~\ref{pr:m7s1}. This is due to the fact that the corresponding systems of Diophantine equations are almost identical. Indeed, in this case we have
        \begin{align}
        50+2s & = \sum a_i^2 + 6\sum b_i^2 + s\sum c_i^2 + 6s\sum d_i^2,\label{eq:6s1:1}\\
        4 & = \sum a_i b_i + s\sum c_i d_i,\label{eq:6s1:2} \\
        4 & = \sum a_i c_i + 6\sum b_i d_i,\label{eq:6s1:3}\\
        0 & = \sum a_i d_i + \sum b_i c_i,\label{eq:6s1:4}
        \end{align}
      with the same congruence conditions $a_i\equiv c_i$ and $b_i\equiv d_i \pmod{2}$.
    \end{proof}


    \begin{proposition}\label{pr:m6s2}
      Let $K=\K{6}{s}$ for $s \equiv 2 \pmod{4}$, $s\neq 14$. Then $\P(\O_K)\geq 6$.

      In particular, the following element has length $6$:
        \[
        \alpha = \begin{cases}
          1^2 + 1^2 + 1^2 + (1+\sqrt{6})^2 + \Bigl(\frac{\sqrt{6}+\sqrt{s}}{2}\Bigr)^2 + \Bigl(1+\frac{\sqrt{6}+\sqrt{s}}{2}\Bigr)^2 & \text{ if } s\geq 22,\\
          1^2 + 1^2 + 1^2 + (2+\sqrt{6})^2 + \Bigl(\frac{\sqrt{6}+\sqrt{10}}{2}\Bigr)^2 + \Bigl(1+\frac{\sqrt{6}+\sqrt{10}}{2}\Bigr)^2 & \text{ if } s=10.
        \end{cases}
        \]
    \end{proposition}
           
    \begin{proof}
      This particular case is the only one when $m\in\{6,7\}$ and $s$ is not coprime with $m$. In fact, $\gcd(6,s)=2$. For this proof, we will slightly change the notation. Let $z\equiv 1\pmod{2}$ satisfy $s=2z$, then $t=3z$. As Lemma~\ref{lem:m67pc} takes care of $s\leq 110$ (and $s$ is not divisible by $3$), we may assume in the following that $s\geq 118$ (hence $z\geq 59$).

      Moreover, we have to distinguish two cases for $z\bmod 4$, as that influences the value of $t=3z$ modulo $4$ and, consequently, the integral basis of $\K{6}{s}$.

      \begin{enumerate}[label= (\alph*)]
        \item $z\equiv 1 \pmod{4}$: 
        Suppose that $\alpha=\sum x_i^2$ for $x_i\in \O_K$. As $K$ is of type (B1) and the integral basis is $\Bigl(1,\sqrt{6}, \sqrt{3z}, \frac{\sqrt{6}+\sqrt{2z}}{2}\Bigr)$, we may write
          \[
          \alpha = (14 + z) + 3\sqrt{6} + \sqrt{2z} + 2\sqrt{3z} = \sum x_i^2 = \sum\Bigl(\frac{a_i + b_i\sqrt{6} + c_i\sqrt{2z} + d_i\sqrt{3z}}{2}\Bigr)^2,
          \]
        where all $a_i\equiv d_i\equiv 0$ and $b_i\equiv c_i \pmod{2}$.

        By expanding the squares and comparing coefficients, one gets the following conditions:
          \begin{align}
          56+4z & = \sum a_i^2 + 6\sum b_i^2 + 2z\sum c_i^2 + 3z\sum d_i^2,\label{eq:6s2:1}\\
          6 & = \sum a_i b_i + z\sum c_i d_i,\label{eq:6s2:2} \\
          2 & = \sum a_i c_i + 3\sum b_i d_i,\label{eq:6s2:3}\\
          4 & = \sum a_i d_i + 2\sum b_i c_i.\label{eq:6s2:4}
          \end{align}

        If we assume nonzero $\sum d_i^2$, then it has to be at least equal to $4$ and~\eqref{eq:6s2:1} leads to a contradiction. Hence $d=(\zeros)$. From~\eqref{eq:6s2:2},~\eqref{eq:6s2:3} and~\eqref{eq:6s2:4}, one immediately sees that none of $a,b,c$ can be $(\zeros)$, so $\sum a_i^2 \geq 4$ and $\sum b_i^2, \sum c_i^2 \geq 1$. From~\eqref{eq:6s2:1} we then get that $\sum c_i^2\leq 2$. As the assumption that $c=(1,\zeros)$ leads together with $b_i\equiv c_i\pmod{2}$ to a contradiction in~\eqref{eq:6s2:4}, we arrive at
          \[
          d=(\zeros), c=(1,1,\zeros), b=(\text{odd},\text{odd},\evens), a=(\evens).
          \]

        We have from~\eqref{eq:6s2:1} that $56 = \sum a_i^2 + 6\sum b_i^2$, which leads to all $|b_i| \leq 2$. Then~\eqref{eq:6s2:4} yields $b_1=b_2=1$. Either all other $b_i=0$, or there is exactly one more nonzero $b_j$, with $|b_j|=2$. However, if we allowed $b=(1,1,\zeros)$, then~\eqref{eq:6s2:2} and~\eqref{eq:6s2:3} would contradict each other. Thus we arrive at
          \[
          d=(\zeros), c=(1,1,\zeros), b=(1,1,2,\zeros), a=(\evens)
          \]
        and it remains to deal with $a$, satisfying $\sum a_i^2=20,\ a_1+a_2+2a_3=6,\ a_1+a_2=2$. This necessarily leads to $a=(2,0,2,2,2,2,\zeros)$. After rewriting back $2z=s$, we get
          \[
          \alpha = \Bigl(1+\frac{\sqrt{6}+\sqrt{s}}{2}\Bigr)^2 + \Bigl(\frac{\sqrt{6}+\sqrt{s}}{2}\Bigr)^2 + (1+\sqrt{6})^2 + 1^2 + 1^2 + 1^2
          \]
        and $\alpha$ is indeed a sum of $6$ squares in $\O_K$, but not less.

        \item $z\equiv 3 \pmod{4}$: $K$ is now of type (B2) and the integral basis is $\Bigl(1,\sqrt{6}, \frac{1+\sqrt{3z}}{2}, \frac{\sqrt{6}+\sqrt{2z}}{2}\Bigr)$. Our $\alpha$ stays the same and, in fact, we get the same system of equations (\ref{eq:6s2:1},\ref{eq:6s2:2},\ref{eq:6s2:3},\ref{eq:6s2:4}) as in the previous case. Only the congruence conditions on $a_i, b_i, c_i, d_i$ are different (weaker). We now have that all $a_i\equiv d_i$ and $b_i\equiv c_i \pmod{2}$.
        
        $z\geq 59$ together with~\eqref{eq:6s2:1} implies that $\sum d_i^2\leq 1$. Suppose that $d=(1,\zeros)$. First observe that necessarily $c=(\zeros)$, because $\sum c_i^2\geq 1$ together with $\sum d_i^2 = 1$ would again lead to a contradiction in~\eqref{eq:6s2:1}. But $d=(1,\zeros)$ and $c=(\zeros)$ leads to a contradiction in~\eqref{eq:6s2:3}. Therefore, $d=(\zeros)$, $a=(\evens)$ and we can finish the proof by following exactly the same steps as in the case (a), since the congruence conditions became the same.\qedhere
      \end{enumerate}
    \end{proof}


    \begin{proposition}\label{pr:m6s3}
      Let $K=\K{6}{s}$ for $s \equiv 3 \pmod{4}$. Then $\P(\O_K)\geq 6$.

      In particular, the following element has length $6$:
        \[
        \alpha = 1^2 + 1^2 + 1^2 + (1+\sqrt{6})^2 + \Bigl(\frac{\sqrt{6}+\sqrt{6s}}{2}\Bigr)^2 + \Bigl(1+\frac{\sqrt{6}+\sqrt{6s}}{2}\Bigr)^2.
        \]
    \end{proposition}

    \begin{proof}
      Thanks to Convention~\ref{con:mst} and Lemma~\ref{lem:m67pc}, we may assume $s\geq 19$. Suppose that $\alpha=\sum x_i^2$ for $x_i\in \O_K$. As $K$ is of type (B1) and the integral basis is $\Bigl(1,\sqrt{6}, \sqrt{s}, \frac{\sqrt{6}+\sqrt{6s}}{2}\Bigr)$, we may write
        \[
        \alpha = (14 + 3s) + 3\sqrt{6} + 6\sqrt{s} + 6\sqrt{6s} = \sum x_i^2 = \sum\Bigl(\frac{a_i + b_i\sqrt{6} + c_i\sqrt{s} + d_i\sqrt{6s}}{2}\Bigr)^2,
        \]
      where all $a_i\equiv c_i\equiv 0$ and $b_i\equiv d_i \pmod{2}$.

      By expanding the squares and comparing coefficients, one gets the following conditions:
        \begin{align}
        56+12s & = \sum a_i^2 + 6\sum b_i^2 + s\sum c_i^2 + 6s\sum d_i^2,\label{eq:6s3:1}\\
        6 & = \sum a_i b_i + s\sum c_i d_i,\label{eq:6s3:2} \\
        12 & = \sum a_i c_i + 6\sum b_i d_i,\label{eq:6s3:3}\\
        2 & = \sum a_i d_i + \sum b_i c_i.\label{eq:6s3:4}
        \end{align}

      From~\eqref{eq:6s3:1}, we obtain that $\sum d_i^2 \leq 2$. Let us assume that either $d=(\zeros)$ or $d=(1,\zeros)$. If we take into account that $b_i\equiv d_i\pmod{2}$ and look at equations~\eqref{eq:6s3:2} or~\eqref{eq:6s3:3} modulo $4$, respectively, we will get a contradiction. Therefore, $d=(1,1,\zeros)$ and $b=(\text{odd},\text{odd},\evens)$.

      If $c\neq(\zeros)$, then $\sum c_i^2\geq 4$ and~\eqref{eq:6s3:1} leads to a contradiction. Therefore we arrive at
        \[
        d=(1,1,\zeros), c=(\zeros), b=(\text{odd},\text{odd},\evens), a=(\evens).
        \]

      We have from~\eqref{eq:6s3:1} that $56 = \sum a_i^2 + 6\sum b_i^2$, which together with $\sum a_i^2\equiv 0\pmod{4}$ leads to $\sum b_i^2\leq 8$ and all $|b_i| \leq 2$. Then~\eqref{eq:6s3:3} gives us $b_1=b_2=1$. Either all other $b_i=0$, or there is exactly one more nonzero, with $|b_j|=2$. However, if $b=(1,1,\zeros)$, then~\eqref{eq:6s3:2} and~\eqref{eq:6s3:4} would contradict each other. Thus we arrive at
        \[
        d=(1,1,\zeros), c=(\zeros), b=(1,1,2,\zeros), a=(\evens)
        \]
      and the remainder of the proof is similar to the previous ones. As $a$ has to satisfy $\sum a_i^2=20,\ a_1+a_2+2a_3=6,\ a_1+a_2=2$, this necessarily leads to $a=(2,0,2,2,2,2,\zeros)$.  Hence we get
        \[
        \alpha = \Bigl(1+\frac{\sqrt{6}+\sqrt{6s}}{2}\Bigr)^2 + \Bigl(\frac{\sqrt{6}+\sqrt{6s}}{2}\Bigr)^2 + (1+\sqrt{6})^2 + 1^2 + 1^2 + 1^2
        \]
      and $\alpha$ is indeed a sum of $6$ squares in $\O_K$, but not less.
    \end{proof}
  

  \subsection{Fields with $m=3$}\label{subsec:m3}
    
    For the fields of the form $K=\K{3}{s}$, we failed to prove $\P(\O_K)\geq 6$ in general. However, even in this case we managed to gather a lot of evidence towards the validity of Conjecture~\ref{conj:final}. This is contained in Propositions~\ref{pr:m3s1}, \ref{pr:m3s2} and \ref{pr:m3s3}, which are based on cleverly guessing three families of potential witnesses of length $6$ and verifying them on a computer for (finitely) many fields. By this we extend the results of~\cite{KraRasSga22} where the authors dealt with fields with $s\in\{17,22\}$, $26 \leq s \leq 55$ and $55 < s \leq 79$ odd.

    Before presenting these partial results, let us discuss why this case is so different from the previous ones (with $m\in\{6,7\}$). A significant part of the reason lies in the fact that $\P(\O_{\Q(\sqrt{3})})=3$ while $\P(\O_{\Q(\sqrt{6})})=\P(\O_{\Q(\sqrt{7})})=4$.
    
    Namely, for $m\in\{6,7\}$, we followed the usual strategy of choosing
      \[
      \alpha=\alpha_0+u^2+v^2,
      \]
    where $\ell(\alpha_0)=4$ in $\O_{\Q(\!\sqrt{m})}$. This gave us $\alpha\in\O_K$ with $\ell(\alpha)=6$, as we proved in Propositions~\ref{pr:m7s1}--\ref{pr:m6s3}. In case $m=3$, it is similarly easy to choose $\alpha_0$ with $\ell(\alpha_0)=3$ in $\O_{\Q(\!\sqrt{3})}$ and add two more squares to obtain $\alpha\in\O_K$ with $\ell(\alpha)=5$, see~\cite[Section 6.5]{KraRasSga22}. But in order to get $\ell(\alpha)=6$, we need one more square, and that is the point where things get challenging. In particular, it proved surprisingly difficult to obtain an element with $\ell(\alpha)=6$ from the known elements of length $5$ by just adding one more square. Nevertheless, we managed to find several families of elements (not of the form $\alpha_0+u^2+v^2+w^2$ with $\alpha_0\in\O_{\Q(\!\sqrt{3})}$) which we verified to indeed have length $6$ for multiple $\K{3}{s}$ (with rather small $s$). 
    
    In the following, we explicitly write down only those elements of length $6$ which belong to one of these families, while omitting the few ones which are of the different form and only ``fill the gaps''. All these families have quite a complicated form which makes use of the ``odd floor'' and ``odd ceiling'' functions, denoted here by $\flfl{x}$ and $\cece{x}$, assigning to each real $x$ the largest odd integer $n \leq x$ and the smallest odd integer $n \geq x$, respectively. One could see a certain similarity to one known family of elements with $\ell(\alpha)=5$ in $\K{5}{s}$, see \cite[Section 7.1]{KraRasSga22}.\footnote{The section title there, ``Lower bound in the most difficult case'', nicely illustrates the challenge involved.}

    \begin{remark}
      Conjecture~\ref{conj:final} predicts that $\P(\O_K)\geq 6$ for $K=\K{3}{s}$, $s>14$. We hope that the families presented below (or similar ones) may indeed be used to prove this (or at least the weaker statement that $\P(\O_K)\geq 6$ for infinitely many $K=\K{3}{s}$). 
      
      The upper bounds for $s$ in the following Propositions are due to the necessary computation time. The computations needed for the verification of family in Proposition~\ref{pr:m3s1} turned out to be exceptionally fast, which allowed us to handle much higher values of $s$ than in Propositions~\ref{pr:m3s2} and \ref{pr:m3s3}. Note that for several odd numbers, namely $315$, $999$, $3161$ and $3163$, we covered all fields $\K{3}{s}$ with this value of $\flflsq{s}$.\footnote{These values were chosen so that the corresponding intervals for $s$ contain $10^5$, $10^6$ and $10^7$.} This indicates that Proposition~\ref{pr:m3s1} possibly holds for all $s\equiv 1\pmod{4}, s\geq 17$. 
      
      On the other hand, the family from Proposition~\ref{pr:m3s2} is much weaker -- it happens for several odd numbers $n$ that among all fields with $n^2 < s < (n+2)^2$ there is at least one exception with $\ell(\alpha) < 6$.
    \end{remark}

    \begin{proposition}\label{pr:m3s1}
      Let $K=\K{3}{s}$ for $s \equiv 1 \pmod{4}$. If $17 \leq s < 15000$, $315^2 < s < 317^2$, $999^2 < s < 1001^2$ or $3161^2 < s < 3165^2$, then $\P(\O_K)\geq 6$.

      In particular:
      \begin{itemize}
        \item If $s=29$, then there exists $\alpha\in\O_K$ with $\ell(\alpha)=6$ and $\Tr(\alpha)=314$.
        \item For all remaining $s$, the following $\alpha$ has $\ell(\alpha)=6$:        
          \begin{align*}
          \alpha = &\ 1^2 + \Bigl(\frac{\flflsq{s}+\sqrt{s}}{2}\Bigr)^2 + \Bigl(\frac{\flflsq{s}+\sqrt{s}}{2}\Bigr)^2 + \Bigl(\frac{\cecesq{s}+\sqrt{s}}{2}\Bigr)^2 + \\
          & + \Bigl(\frac{(1+\sqrt{3})\cdot(\cecesq{s}+\sqrt{s})}{2}\Bigr)^2 + \Bigl(1 - \frac{\sqrt{3}\cdot(\flflsq{s}+\sqrt{s})}{2}\Bigr)^2.
          \end{align*}          
      \end{itemize}
    \end{proposition}

    \begin{proposition}\label{pr:m3s2}
      Let $K=\K{3}{s}$ for $s \equiv 2 \pmod{4}$. If $22 \leq s \leq 506$ and does not belong to the set $\{170,178,230,238,362,442,446,454\}$, then $\P(\O_K) \geq 6$. 

      In particular:
      \begin{itemize}
        \item If $22 \leq s \leq 46$, then there exists $\alpha\in\O_K$ with $\ell(\alpha)=6$ and $\Tr(\alpha)<1000$.
        \item If $s\in\{62,70,74\}$, then there exists $\alpha\in\O_K$ with $\ell(\alpha)=6$ and $\Tr(\alpha)$ equal to $1260$, $1388$ and $1512$, respectively.

        \item For all remaining $s$, the following $\alpha$ has length $6$:        
          \begin{align*}
          \alpha = &\ (2+\sqrt{3})^2 + (2+\sqrt{3})^2 + \Bigl(\frac{(\flflsq{s}-5)-2\sqrt{3}-\sqrt{s}-\sqrt{3s}}{2}\Bigr)^2 + \Bigl(\frac{(\flflsq{s}-3)-\sqrt{s}-\sqrt{3s}}{2}\Bigr)^2 + \\
          & + \Bigl(\frac{(\flflsq{s}-9)-4\sqrt{3}-\sqrt{s}-\sqrt{3s}}{2}\Bigr)^2 + \Bigl(\frac{(\flflsq{s}+1)+\sqrt{3}\cdot(3-\flflsq{s})+2\sqrt{s}}{2}\Bigr)^2.
          \end{align*}
      \end{itemize}
    \end{proposition}

    \begin{proposition}\label{pr:m3s3}
      Let $K=\K{3}{s}$ for $s \equiv 3 \pmod{4}$. If $19 \leq s \leq 511$, then $\P(\O_K) \geq 6$. 

      In particular:
      \begin{itemize}
        \item If $19 \leq s \leq 91$, then there exists $\alpha\in\O_K$ with $\ell(\alpha)=6$ and $\Tr(\alpha)<1000$.
        \item If $s=119$,  there exists $\alpha\in\O_K$ with $\ell(\alpha)=6$ and $\Tr(\alpha)=1098$.
        \item For all remaining $s$, the following $\alpha$ has length $6$:        
          \begin{align*}
          \alpha = &\ 1^2 + 1^2 + \Bigl(\frac{(\flflsq{s}-5)+3\sqrt{3}+\sqrt{s}}{2}\Bigr)^2 + \Bigl(\frac{(\flflsq{s}-7)+5\sqrt{3}+\sqrt{s}}{2}\Bigr)^2 + \\
          & + \Bigl(\frac{(\flflsq{s}-5)+5\sqrt{3}+\sqrt{s}}{2}\Bigr)^2 + \Bigl(\frac{(22-\flflsq{s})+\sqrt{3}\cdot(\flflsq{s}-12)-\sqrt{s}+\sqrt{3s}}{2}\Bigr)^2.
          \end{align*}
      \end{itemize}
    \end{proposition}


\section{Open problems}\label{sec:open}

  In order to fully divide all biquadratic fields into two classes with $\P(\O_K) \leq 5$ and $\P(\O_K) \geq 6$ (and thus prove both the Conjecture~\ref{conj:final} and the weaker Conjecture 1.6(1) from~\cite{KraRasSga22}), the following tasks remain. For one, to find infinite families of elements $\alpha$ in fields $\K{3}{s}$ with $\ell(\alpha)\geq 6$ (we suggest some candidate families in Section~\ref{subsec:m3}). For another, to prove that the six exceptional fields from Conjecture~\ref{conj:final} indeed have $\P(\O_K)\leq 5$.

  Conjecture 1.6(3) from~\cite{KraRasSga22} also remains unresolved. It predicted that $\K{2}{3}$, $\K{2}{5}$, $\K{3}{5}$ have $\P(\O_K)= 3$ and $\K{2}{7}$, $\K{3}{7}$, $\K{5}{6}$, $\K{5}{7}$ have $\P(\O_K)= 4$. The solution for $\K{2}{5}$ has already been announced by Krásenský and Scharlau, but proving the predicted upper bound for the remaining six fields remains an open problem.

  This is connected to the possible improvements of the upper bounds for $\P(\O_K)$. The best general estimate for biquadratic fields is $\P(\O_K)\leq 7$ (see ~\cite{KalYat21}). However, if the conjecture that $g_{\Z[\sqrt{3}]}(2)=6$ from~\cite[Section 1]{KraYat23} was proven, it would immediately lead to $\P(\O_K)\leq 6$ for all $K=\K{3}{s}$.

  A natural next step is to determine which biquadratic fields have $\P(\O_K)=6$ and which have $\P(\O_K)=7$. Some infinite families with $\P(\O_K)=7$ were already found in~\cite{Tin25} and \cite{KraRasSga22}. Interestingly enough, $\K{6}{7}$ and $\K{13}{15}$ have $\P(\O_K)$ equal to $7$ (see~\cite[Lemma 5.2]{KraYat23}). This seems to indicate that $7$ might be the more typical value.


\section*{Acknowledgement}

  I would like to thank Jakub Krásenský for introducing me to the topic of this paper in the first place and for several helpful discussions during its early stages. I am also grateful for his insightful comments on an earlier version of the manuscript.

\bibliography{reference}{}

\begin{thebibliography}{DGMY25}

\bibitem[BCP97]{BosCanPla97}
Wieb Bosma, John Cannon, and Catherine Playoust.
\newblock The magma algebra system i: The user language.
\newblock {\em Journal of Symbolic Computation}, 24(3):235--265, 1997.

\bibitem[DGMY25]{DaaGajSiuYat25}
Nicolas Daans, Stevan Gajović, Siu~Hang Man, and Pavlo Yatsyna.
\newblock Pythagoras numbers for infinite algebraic fields.
\newblock Preprint. Available at arXiv:2502.11222, 2025.

\bibitem[GMnT25]{GilTin25}
Daniel Gil-Mu\~noz and Magdaléna Tinková.
\newblock Additive structure of non-monogenic simplest cubic fields.
\newblock {\em Ramanujan J.}, 66(3):Paper No. 47, 56, 2025.

\bibitem[HH22]{HeHu22}
Zilong He and Yong Hu.
\newblock Pythagoras number of quartic orders containing $\sqrt{2}$.
\newblock Preprint. Available at arXiv:2204.10468v2, 2022.

\bibitem[Kr{\'a}22]{Kra22}
Jakub Kr{\'a}senský.
\newblock A cubic ring of integers with the smallest {P}ythagoras number.
\newblock {\em Arch. Math. (Basel)}, 118(1):39--48, 2022.

\bibitem[Kr{\'a}23]{Kra23}
Jakub Kr{\'a}senský.
\newblock {\em Universal quadratic forms over orders in number fields}.
\newblock Doctoral thesis, Univerzita Karlova, Matematicko-fyzikální fakulta,
  Prague, Czech Republic, 2023.

\bibitem[KRS22]{KraRasSga22}
Jakub Krásenský, Martin Raška, and Ester Sgallová.
\newblock Pythagoras numbers of orders in biquadratic fields.
\newblock {\em Expo. Math.}, 40(4):1181--1228, 2022.

\bibitem[KST25]{KalSgaTin25}
Vítězslav Kala, Ester Sgallová, and Magdaléna Tinková.
\newblock Arithmetic of cubic number fields: {J}acobi-{P}erron, {P}ythagoras,
  and indecomposables.
\newblock {\em J. Number Theory}, 273:37--95, 2025.

\bibitem[KTZ20]{KraTinZem20}
Jakub Krásenský, Magdaléna Tinková, and Kristýna Zemková.
\newblock There are no universal ternary quadratic forms over biquadratic
  fields.
\newblock {\em Proc. Edinb. Math. Soc. (2)}, 63(3):861--912, 2020.

\bibitem[KY21]{KalYat21}
Vítězslav Kala and Pavlo Yatsyna.
\newblock Lifting problem for universal quadratic forms.
\newblock {\em Adv. Math.}, 377:Paper No. 107497, 24, 2021.

\bibitem[KY23]{KraYat23}
Jakub Krásenský and Pavlo Yatsyna.
\newblock On quadratic {W}aring's problem in totally real number fields.
\newblock {\em Proc. Amer. Math. Soc.}, 151(4):1471--1485, 2023.

\bibitem[Tin23]{Tin23}
Magdaléna Tinková.
\newblock On the {P}ythagoras number of the simplest cubic fields.
\newblock {\em Acta Arith.}, 208(4):325--354, 2023.

\bibitem[Tin25]{Tin25}
Magdaléna Tinková.
\newblock Bounds on the {P}ythagoras number and indecomposables in biquadratic
  fields.
\newblock {\em Proc. Edinb. Math. Soc.}, pages 1--26, 2025.

\bibitem[Wil70]{Wil70}
Kenneth~S. Williams.
\newblock Integers of biquadratic fields.
\newblock {\em Canad. Math. Bull.}, 13:519--526, 1970.

\end{thebibliography}
\bibliographystyle{alpha}

\end{document}